\documentclass[11pt,leqno]{amsart}
\usepackage{fullpage}
\usepackage{amsmath}
\usepackage{amsfonts}
\usepackage{amssymb}
\usepackage{amsthm}
\usepackage{graphicx,xy}
\usepackage{enumerate}
\DeclareMathAlphabet{\mathpzc}{OT1}{pzc}{m}{it}
\def\kfield{\mathbf k}

\def\SC{{\mathcal{SC}}}
\def\Sing{C\hsc}
\def\hsc{{\mathrm{\bf{sc}}}}
\def\P{{\mathcal{P}}}
\def\C{{\mathcal{C}}}

\def\R{{\mathbb R}}

\def\cl{\mathpzc c}
\def\op{\mathpzc o}

\numberwithin{equation}{section}

\newtheorem{thm}{Theorem}[section]
\newtheorem*{thm*}{Theorem}
\newtheorem{prop}[thm]{Proposition}
\newtheorem{cor}[thm]{Corollary}
\theoremstyle{definition}
\newtheorem{defn}[thm]{Definition}

\theoremstyle{remark}

\theoremstyle{remark}

\xyoption{all}
\CompileMatrices

\begin{document}

\title[Non-formality of the Swiss-Cheese operad]{Non-formality of the Swiss-Cheese operad}
\author{Muriel Livernet}
\address{Universit\'e Paris 13, Sorbonne Paris Cit\'e, LAGA, CNRS, UMR 7539, 93430 
  \mbox{Villetaneuse}, France}
\email{livernet@math.univ-paris13.fr}

\date{\today}

\maketitle

%\vspace{-0.5cm}
\begin{abstract}
In this note, we prove that the Swiss-cheese operad is not formal. We also give a criteria in terms of Massey operadic product for the non-formality
of a topological operad.
\end{abstract}

\section*{Introduction}

The question of formality finds its roots in rational homotopy theory. In \cite{Sullivan77}, D. Sullivan proves that the rational homotopy type of a space  
is entirely determined by a commutative differential graded algebra. The space is formal if this algebra is 
quasi-isomorphic to its homology,  
that is, if the rational homotopy type of the space is entirely determined by its cohomology ring. This notion has been applied with 
great success in \cite{DGMS}, where the authors 
prove that the real homotopy type of a simply connected compact K\"ahler manifold is entirely determined by its cohomology ring.
The question of non-formality has its own interest, in particular in the study of symplectic manifolds. In such cases, one can ask whether 
a symplectic manifold admits a  K\"ahler structure or not.  A symplectic manifold which is 
not formal is certainly not a K\"ahler manifold (see e.g. \cite{FM08}).

More generally,  formality is closely related to deformation of structures. And structures are ruled by operads. One of the most famous 
topological operad is the little $d$-discs operad of Boardman and Vogt \cite{BV}, recognizing iterated loop spaces (see \cite{May}). 
An operad is formal if its singular chain complex is an operad quasi-isomorphic to its homology operad.  Algebras over the homology operad of 
the $d$-little discs operad are $d$-Gerstenhaber algebras.
In \cite{Kontse99}, M. Kontsevich lays the groundwork for the formality of the little $d$-discs operad and its applications (see the introduction of 
Giansiracusa and Salvatore in \cite{GS12} for a history of the proof of the formality of the little
$d$-discs operad). The Swiss-cheese operad appears in the work of Voronov in \cite{Voronov99} and Kontsevich in \cite{Kontse99} as a tool to handle 
actions between $d$-Gerstenhaber algebras and $(d-1)$-Gerstenhaber algebras and their deformations. 

\medskip

In this note, we prove that the  
Swiss-cheese operad $\SC_d$ is not formal for any $d\geq 2$. Let us comment some related results. For $d=2$, Dolgushev studies in \cite{Dolgushev11} 
the first sheet of
the homology spectral sequence for the Fulton-MacPherson version of the
Swiss-cheese operad and proves that it is not formal. However,  this does not imply the non-formality of $\SC_2$, because we proved with E. Hoefel in
\cite{HoeLiv13} that the homology spectral sequence, though collapsing at page 2, does not converge as an operad, to the homology operad of $\SC_2$.
Lambrechts and Volic study the inclusion of the little $n$-discs operad into the little $m$-discs operad and prove in \cite{LV08}, that the inclusion is 
formal
for $m>2n$. Turchin and Willwacher show in \cite{TurWill} that this inclusion is not formal if $m=n+1$. Their proof is based on Kontsevich's 
graph complex.  There might be a link between the deformations of the aforementioned inclusion and the
deformations of the Swiss-cheese operad, but this is not clear. If there is,
the technics involved in our paper are much simpler because they only use basic algebraic topology. Indeed, 
we prove the non-formality by building non-vanishing Massey operadic products. 
The first three sections of the note focus on the $2$-dimensional Swiss-cheese 
operad, whereas the fourth section treats the $d$-dimensional case. The last section is of independent interest and 
gives a criteria for non-formality of a topological operad using
Massey operadic products.

\medskip

\section{The two-dimensional Swiss-cheese operad}

The Swiss-cheese operad is a two-colored operad that has been defined by Voronov in \cite{Voronov99}.  There is another version of the Swiss-cheese operad given by Kontsevich in \cite{Kontse99}. The homology of these two topological colored operad has been studied by Hoefel and the author in
 \cite{HoeLiv12} and \cite{HoeLiv13}.  In this note, we consider the Kontsevich version of the Swiss-cheese operad  denoted by 
 $\SC$ as in \cite{HoeLiv13}.
It is a two-colored topological operad, whose colors are the closed color $\cl$ and the open color $\op$. 

\smallskip

The topological space $\SC(x_1,\ldots,x_k;\cl)$ is the empty space if $k=0$ or if $\exists i,x_i=\op$ and 
the  ordered configuration space of $k$ nonoverlapping discs  in the unit disc in the plane otherwise.

For a $k$-tuple $(x_1,\ldots,x_k)\in \{\op,\cl\}^k$, let $n$ (resp. $m$) be the number of $x_i$'s such that $x_i=\cl$ (resp. $x_i=\op$).
The topological space $\SC(x_1,\ldots,x_k;\op)$ is the empty space if $k=0$ and otherwise the ordered configuration space of  
nonoverlapping $n$ discs and $m$ semidiscs  in the unit  upper semidisc in the plane, the semidiscs being centered 
on the real line. With our convention $m$ can be zero, that is, we allow  operations having only closed inputs and an open output.

\smallskip
Any colored operad comes equipped with an action of the symmetric groups, that is, if $\P$ is an $I$-colored operad, 
there is an isomorphism $\psi_\sigma:\P(x_1,\ldots,x_n;y)\rightarrow \P(x_{\sigma(1)},\ldots,x_{\sigma(n)};y)$, for every $\sigma\in S_n$,
$x_1,\ldots,x_n,y\in I$. In the sequel we use the notation $\psi_\sigma(p)=p\cdot\sigma$.

The topological space obtained from $\SC$ by considering the
 $n$ first inputs of color $\cl$, the last $m$ inputs of color $\op$ with the output of color $x\in\{\op,\cl\}$ is denoted by $\SC(n,m;x)$. 
As pointed out above,  the other spaces 
 describing $\SC$ are obtained using the action of the symmetric group. For instance $\SC(\op,\cl;\op)=\SC(1,1;\op)\cdot (21)$ where $(21)$ denotes the transposition in the symmetric group of two variables.

\subsection{The homology of the operad $\SC$}
We work over a field $\kfield$. We recall that the singular chain complex over $\kfield$ of a topological operad is an operad in the category of differential graded $\kfield$-vector spaces (or dg operad for short).
The singular chain complex of $\SC$ is denoted by $\Sing$ and its homology is denoted  by $\hsc$.

We recall from \cite[Proposition 3.2.1]{HoeLiv13} that algebras over the operad $\hsc$ are triples $(G,A,h)$ where 
$G$ is a Gerstenhaber algebra, $A$ is an associative algebra and $h:G\rightarrow A$ is a central morphism of associative algebras.

\medskip

\subsection{Notation for the generators of $\hsc$} We use the same notation as in \cite[Corollary 3.2.2]{HoeLiv13}. The operad 
$\hsc$ is an operad generated by $f_2\in\hsc(2,0;\cl)_0$, $g_2\in\hsc(2,0;\cl)_1$, $e_{0,2}\in\hsc(0,2;\op)_0$ and 
$e_{1,0}\in\hsc(1,0;\op)_0$. The operation $f_2$ is the commutative product and the operation $g_2$ is the 
Lie product governing the Gerstenhaber structure. The operation $e_{0,2}$ is the associative product governing 
the associative 
structure, and the operation $e_{1,0}$ is the one governing the central morphism.

\subsection{Notation for operadic composition} Given a $\{\cl,\op\}$-colored operad $\mathcal P$ and $\alpha\in\mathcal P(n,m;x)$ and $\beta\in\mathcal P(r,s;\op)$
 the notation $\alpha\circ_i^\op \beta$ holds for the element  obtained by inserting $\beta$ at the $i$-th open input of $\alpha$. If $\beta\in\mathcal P(r,s;\cl)$, then $\alpha\circ_j^\cl \beta$ is  the element obtained by inserting $\beta$ at the $j$-th closed input of $\alpha$.
Note that $\alpha\circ_i^\op\beta\in \mathcal P(\underbrace{\cl,\ldots,\cl}_{n\text{ times}},\underbrace{\op,\ldots,\op}_{i-1 \text{ times}},\underbrace{\cl,\ldots,\cl}_{r\text{ times}},\underbrace{\op,\ldots,\op}_{s+m-i \text{ times}};x).$

\section{The singular chain complex of the two-dimensional Swiss-cheese operad}\label{S:singular2}

In this section, we  describe specific elements in the singular chain complex $\Sing$ of the Swiss-cheese operad.
For a cycle $x$ in a chain complex $C$, its image in $H_*(C)$ is denoted by $[x]$.

%\medskip

Let $a\in \SC(0,2;\op)$ be the configuration of the two semidiscs of radius $\frac{1}{2}$ whose 
center have respectively coordinates $(-\frac{1}{2},0)$ and 
$(\frac{1}{2},0)$. We denote also by $a$ the corresponding element in $\Sing(0,2;\op)_0.$
Let $f\in \SC(1,0;\op)$ be the configuration of the disc of radius $\frac{1}{2}$ whose center 
has coordinates $(0,\frac{1}{2})$, also seen as an element
in $\Sing(1,0;\op)_0.$

-
\begin{center}
\begin{tabular}{cp{2cm}c}
\includegraphics[height=2cm]{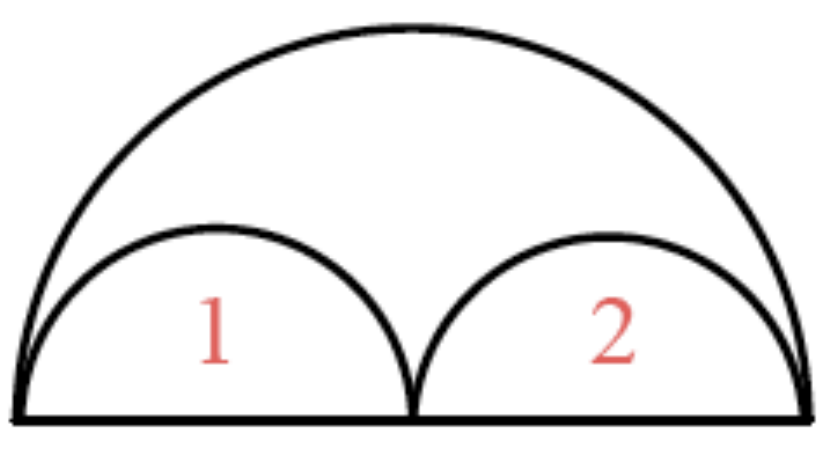}&\quad&
\includegraphics[height=2cm]{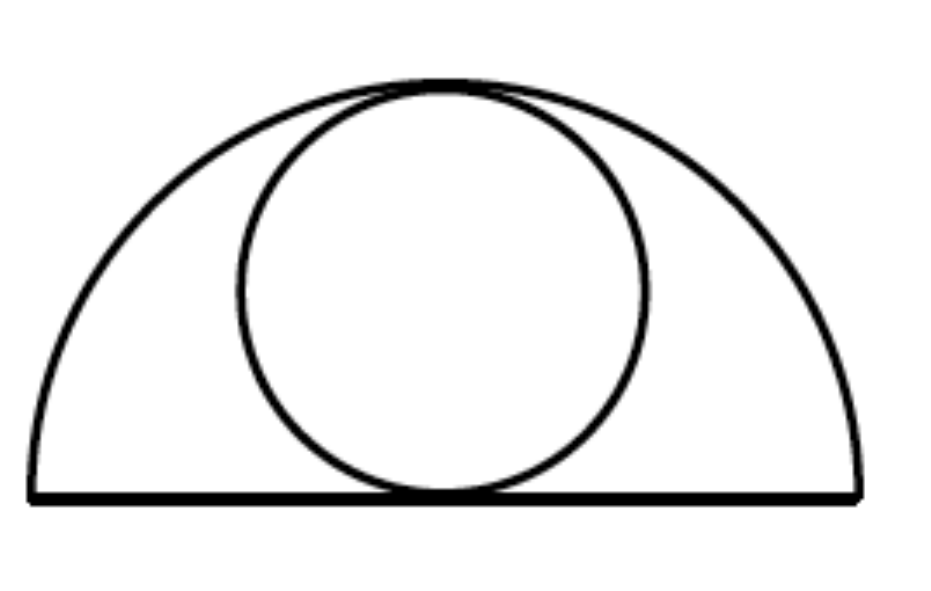}\\
$a\in \SC(0,2;\op)$ &\quad& $f\in \SC(1,0;\op)$ 
\end{tabular}
\end{center}

\medskip

Let $\eta_1\in \Sing(1,1;\op)_1$ be the path in $\SC(1,1;\op)$ 
from $a\circ_1^\op f$ to $(a\circ_2^\op f)\cdot (21)$ such that the center of the left disc runs 
along the circle of radius $\frac{1}{2}$ centered at $(0,\frac{1}{4})$ and the center of the right 
semidisc runs along the real line from the point of coordinates $(\frac{1}{2},0)$ to the point of 
coordinates $(-\frac{1}{2},0)$ as presented in the following picture:

\begin{center}
\begin{tabular}{c}
\includegraphics[height=3cm]{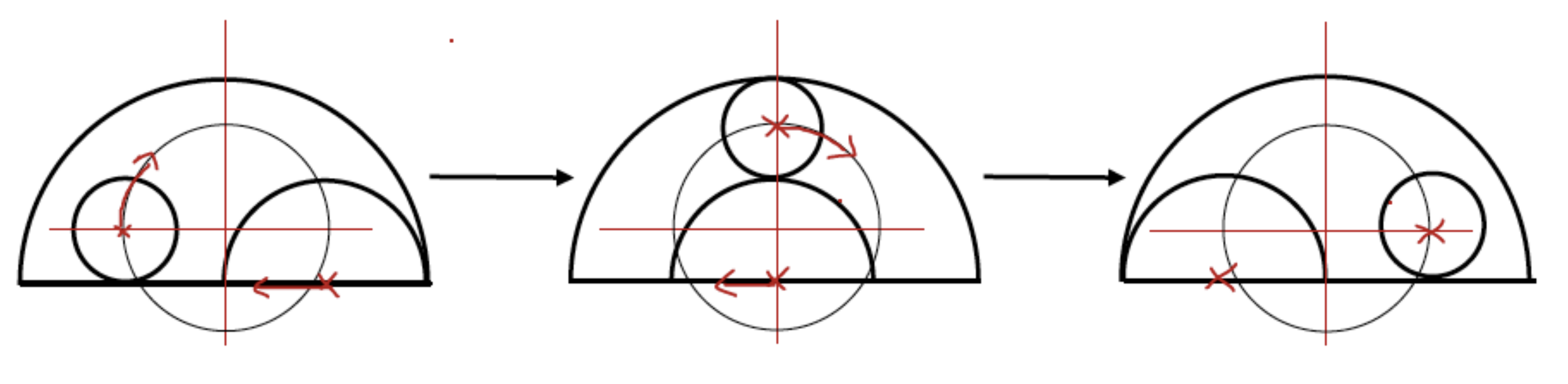}\\
%{\sl The path $\eta_1$}
\end{tabular}
\end{center}

\medskip
The path $\eta_1$ satisfies
\begin{equation}\label{eta1}
\partial\eta_1=(a\circ_2^\op f)\cdot (21)-a\circ_1^\op f.
\end{equation}

%\newpage
Let $l$ be the following loop in $\Sing(2,0;\cl)_1$: 
%be a lift of the generator $g_2$ in $\hsc(2,0;\cl)_1$. We will choose the following loop for $l$.

\begin{center}
\begin{tabular}{c}
\includegraphics[height=8cm]{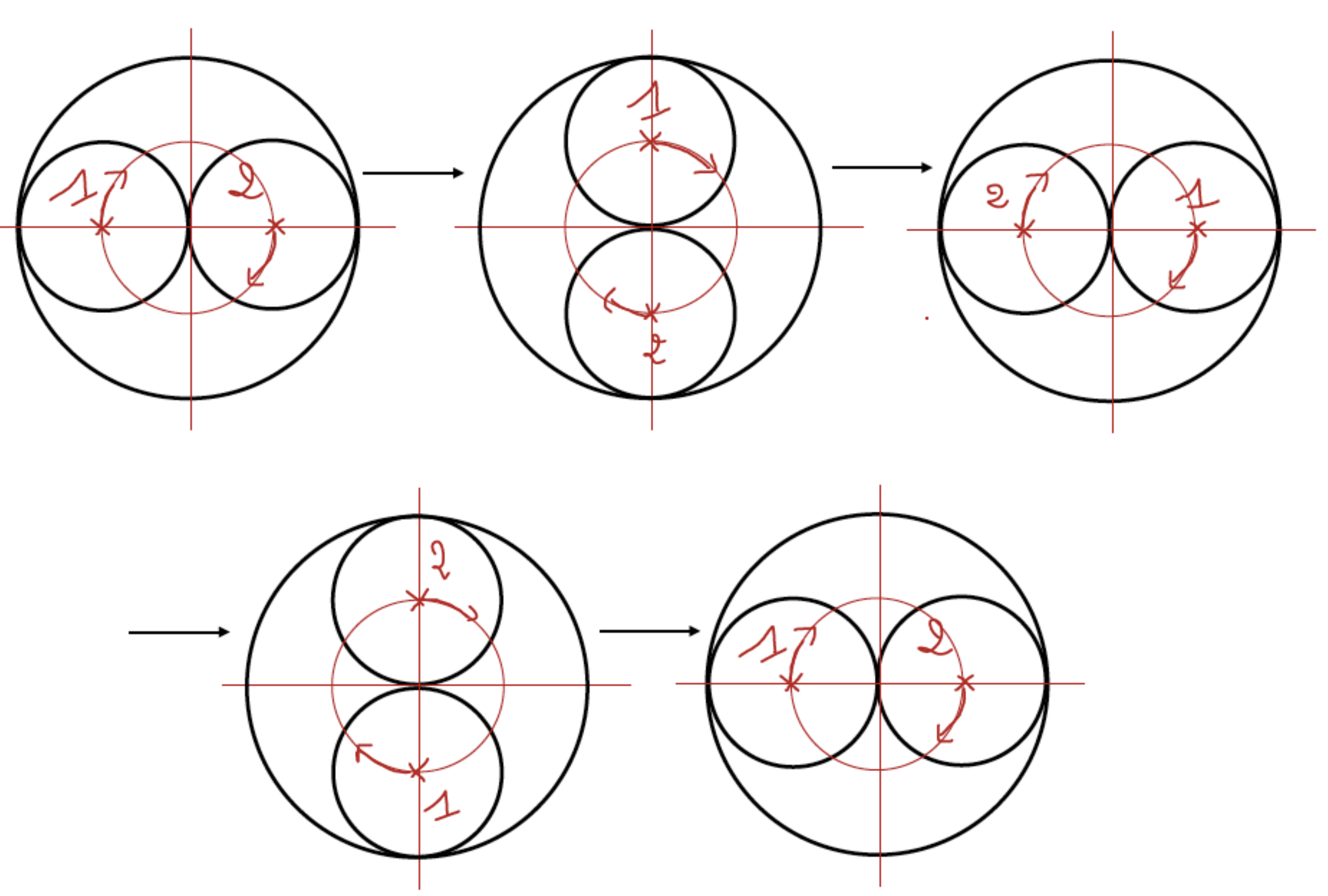}\\
{\sl The loop $l$}
\end{tabular}
\end{center}

\medskip

The chains $a,f,l$ are cycles.
The two dimensional vector space $\hsc(0,2;\op)_0$ is a free $S_2$-module generated by $[a]$. The one dimensional vector spaces 
$\hsc(1,0;\op)_0$ and $\hsc(2,0;\cl)_1$ are spanned respectively by $[f]$ and $[l]$.
Moreover  $[f\circ_1^{\cl} l]=[f]\circ_1^{\cl} [l]$ spans $\hsc(2,0;\op)_1$.

The composite of the paths  $\eta_1\circ_1^\op f$ and $(\eta_1\circ_1^\op f) \cdot (21)$ is a loop 
homotopic to the loop $f\circ_1^\cl l$, hence $[\eta_1\circ_1^\op f+  (\eta_1\circ_1^\op f) \cdot (21)]=[f\circ_1^\cl l]$ and
there exists $\eta_2\in\Sing(2,0;\op)_2$ such that
\begin{equation}\label{eta2}
\partial\eta_2=\eta_1\circ_1^\op f+  (\eta_1\circ_1^\op f) \cdot (21)-f\circ_1^\cl l.
\end{equation}

This relation is dubbed the Eye Law in  \cite{HoeLiv12}, a consequence of the eye shape of the 
compactification of the space of configurations of two points in the upper half plane obtained by Kontsevich in \cite{Kontse03}.

\begin{prop}\label{P:change} If $a'\in\Sing(0,2;\op)_0, f'\in\Sing(1,0;\op)_0$ and $l'\in\Sing(2,0;\cl)_1$ satisfy 
$[a']=[a], [f']=[f]$ and $[l']=[l]$, then there exist $\eta_1'\in \Sing(1,1;\op)_1$ 
and $\eta_2'\in\Sing(2,0;\op)_2$  such that

$$\begin{array}{lcclc}
(\ref{eta1})'&\partial\eta_1'&=&(a'\circ_2^\op f')\cdot (21)-a'\circ_1^\op f',  \\
(\ref{eta2})'& \partial\eta_2'&=&\eta_1'\circ_1^\op f'+  (\eta_1'\circ_1^\op f') \cdot (21)-f'\circ_1^\cl l'.
\end{array}$$
\end{prop}

\begin{proof} There exist $a_1,f_1,l_2$ such that
%\begin{align*}
%a'=&a+\partial a_1\\
%f'=&f+\partial f_1\\
%l'=&l+\partial l_2\\
%\end{align*}
$a'=a+\partial a_1,
f'=f+\partial f_1 
\text{ and } 
l'=l+\partial l_2.$
The elements
\begin{align*}
\eta_1'=&\eta_1+(a\circ_2^\op f_1)\cdot (21)-a\circ_1^\op f_1+(a_1\circ_2^\op f')\cdot(21)-a_1\circ_1^\op f', \\
\eta_2'=&\eta_2-f\circ_1^\cl l_2-f_1\circ_1^\cl l'-(\eta_1\circ_1^\op f_1+a(f_1,f_1))\cdot(\textrm{id}+(21)),
\end{align*}
satisfy the prescribed relations.
\end{proof}

\section{Non-formality of the two-dimensional Swiss-cheese operad $\SC$}

\begin{thm}\label{T:main2} The two-dimensional Swiss-cheese operad is not formal.
\end{thm}

\begin{proof}

We will assume by contradiction that there is a dg operad $\P$ together with a zig-zag of operad quasi-isomorphisms
$$\xymatrix{ & \P \ar[dl]_{\gamma}\ar[dr]^{\psi} & \\
\Sing & &\hsc}$$

Because $H_*(\gamma)$ is an isomorphism, there
are cycles $a^\P,f^\P,l^\P \in \P$ such that $[\gamma(x^\P)]=[x],\forall x\in\{a,f,l\}.$  
There exist $a',f',l'$ as in Proposition \ref{P:change} such that $\gamma(x^\P)=x' ,\forall x\in\{a,f,l\}.$  
By Relation (\ref{eta1})
$$H_*(\gamma)(([a^\P]\circ_2^\op [f^\P])\cdot (21)-[a^\P]\circ_1^\op [f^\P])=0,$$ 
and  there exists $\eta_1^\P\in\P(1,1;\op)_1$ such that 
\begin{equation}\label{eta1P}
\partial\eta_1^\P=(a^\P\circ_2^\op f^\P)\cdot (21)-a^\P\circ_1^\op f^\P.
\end{equation}
Applying $\gamma$  and Proposition \ref{P:change} we have
$$\partial\gamma(\eta_1^\P)=\gamma(\partial\eta_1^\P)=(a'\circ_2^\op f')\cdot (21)-a'\circ_1^\op f'=\partial \eta_1' .$$
Since $\hsc(1,1;\op)$ is concentrated in degree $0$, there exists $x_2\in \Sing(1,1;\op)_2$ such that
$$\partial x_2=\gamma(\eta_1^\P)- \eta_1'.$$

The following relation is a consequence of Relation (\ref{eta1P}) and the operadic relations in $\P$:
$$\partial(\eta_1^\P\circ_1^\op f^\P+(\eta_1^\P\circ_1^\op f^\P)\cdot(21))=0\in\P(2,0;\op)_0.$$
The homology group $\hsc(2,0;\op)_1$ is one-dimensional, spanned by $[f\circ_1^\cl l]=[f'\circ_1^\cl l']$. It implies that 
$H_1(\P(2,0;\op))$ is spanned by $[f^\P\circ_1^\cl l^\P]$ and that there exists $\eta_2^\P\in\P(2,0;\op)_2$ and $\lambda\in\kfield$ such that

\begin{equation}\label{eta2P}
\eta_1^\P\circ_1^\op f^\P+(\eta_1^\P\circ_1^\op f^\P)\cdot(21)=\lambda (f^\P\circ_1^\cl l^\P)+\partial\eta_2^\P.
\end{equation}
Applying $\gamma$ to the above relation gives
$$(\eta_1'+\partial x_2)\circ_1^\op f'+((\eta_1'+\partial x_2)\circ_1^\op f')\cdot(21)=\lambda (f'\circ_1^\cl l')+\partial\gamma(\eta_2^\P),$$
which, together with Relation (\ref{eta2}') gives %the relation in $\Sing(2,0;\op)_1$

%\begin{eqnarray*}
$$\eta_1'\circ_1^\op f'+(\eta_1'\circ_1^\op f')\cdot(21)=\lambda (f'\circ_1^\cl l')+\partial x'_2=f'\circ_1^\cl l'+\partial\eta_2'.$$
%\end{eqnarray*}
The homology class $[f'\circ_1^\cl l']$ spans $\hsc(2,0;\op)_1$, so $\lambda=1$.

For degree reasons $\psi(\eta_1^\P)=0=\psi(\eta_2^\P)$.  Since $\psi$ is a quasi-isomorphism of operads, there exist $\lambda_f,\lambda_l\in \kfield^*$ such that $\psi(f^\P)=\lambda_f e_{1,0}$ and $\psi(l^\P)=\lambda_l g_2$. Applying $\psi$ to Relation (\ref{eta2P}) gives
$$\lambda\lambda_f\lambda_l\; e_{1,0}\circ_1^\cl g_2=0,$$
a contradiction. 
\end{proof}

\section{Non-formality of the Swiss-cheese operad}

In this section we consider the $d$-dimensional Swiss-cheese operad $\SC_d$ where $d\geq 3$.  
The proof of the non-formality of $\SC_d$ is similar to the case $d=2$, but there are some differences. 
First of all, we will consider in this section a field $\kfield$ of characteristic different from $2$. Note that in case the field is of characteristic 2, then
the operad $\SC_d$ is not formal because $\SC_d(\cl,\cl;\cl)$ is not formal as an $S_2$-module.

%\medskip

\subsection{Notation} The ambient space is $\R^d$, the sphere   $S^{d-1}$ is the subspace of $\R^d$ of points of norm $1$. The sphere of dimension $d-2$ is seen as the following subspace of $S^{d-1}\subset\R^d$: 
$$S^{d-2}=\{x=(x_1,\ldots,x_d)\in S^{d-1}| x_d=0 \}.$$
The ball of dimension $d-1$ is identified with the subspace of $S^{d-1}$
$$D^{d-1}=\{x=(x_1,\ldots,x_d)\in S^{d-1}| x_d\geq 0 \}.$$

For $x\in \R^d$ and $r>0$,  the (open) ball centered at $x$ of radius $r$ is denoted by
$B(x;r)$.  If $x_d=0$, the semiball centered at $x$ of radius $r$ is
$$B^+(x;r)=\{y\in \R^d| ||y-x||<r, y_d>0\}.$$

%If $x\in \R^d$ we will often use the notation $x=(\underline x,x_d)$ to emphasize the last coordinate.
We will use the notation $P_\frac{1}{4}$ for the point of coordinates $(0,\ldots,0,\frac{1}{4})$.

Let $r_i$ be the reflection in the hyperplane $H_i=\{x\in\R^d| x_i=0)$ and let $R_{d-1}$ be the composite of the $d-1$ 
reflections $r_1,\ldots,r_{d-1}$.  For a point $x\in\R^d$ we denote by $x\cdot r_i$ the image of $x$ by the reflection $r_i$.

\subsection{The $d$-dimensional Swiss-cheese operad} Similar to the case $d=2$,
the topological space $\SC_d(n,m;\cl)$ is the empty space if $m>0$ or $n=0$ and the configuration space of $n$ nonoverlapping balls  in the unit  ball in 
$\R^d$ otherwise. The topological space $\SC_d(n,m;\op)$ is the empty space if $n=m=0$ and otherwise the configuration space of  
nonoverlapping $n$ balls and $m$ semiballs  in the unit  semiball $B^+(0;1)$ in $\R^d$.

\subsection{The singular chain complex of the $d$-dimensional Swiss-cheese operad}

In this section, as in section \ref{S:singular2}, we provide specific elements in the singular chain complex 
$\Sing_d$ of the $d$-dimensional Swiss-cheese operad. Although the technics used are similar to the case $d=2$, 
there is a main difference that we explain now.

The spaces $\SC_d(1,0;\op)$ and $\SC_d(1,1;\op)$ are contractible. 
The space 
$\SC_d(0,2;\op)$ is homotopy equivalent to $S^{d-2}$ and the spaces $\SC_d(2,0;\cl)$ and $\SC_d(2,0;\op)$ are 
homotopy equivalent to $S^{d-1}$. An algebra over the operad $\hsc_d$ is a triple $(G_1,G_2,h)$ where $G_1$ is a
$d$-Gerstenhaber algebra, $G_2$ is a $(d-1)$-Gerstenhaber algebra and $h:G_1\rightarrow G_2$ is a map of 
commutative algebras, such that
$[h(a),b]_{d-1}=0,\forall a\in G_1, b\in G_2$. Recall that for $d\geq 2$, a $d$-Gerstenhaber algebra is a vector space 
endowed with a commutative product and a Lie bracket of degree $d-1$ which is a derivation with respect to
the product.

For $d=2$, algebras over $\hsc$ are triples $(G,A,h)$ where 
$A$ is an associative algebra and $h:G\rightarrow A$ is a central morphism.  
We should interpret the notion of central morphism as 
$[h(a),b]=0$, where the bracket is the one obtained by antisymmetrizing the associative product.

The main diference between the cases $d=2$  and $d>2$ is based on the above remark.
In this section, we will pick a generator of the bracket 
$[-,-]_{d-1}$ instead of picking a generator of the associative product (compare Relations (\ref{eta1}) and (\ref{etad})).

\begin{prop} There exist cycles $f\in \Sing_d(1,0;\op)_0$, $a\in \Sing_d(0,2;\op)_{d-2}$ and 
$l\in \Sing_d(2,0;\cl)_{d-1}$ such that
$[a]$, $[f]$ and $[l]$ span their respective homology groups, and such that
$a\cdot (21)=(-1)^{d-1} a$ and $l\cdot (21)=(-1)^{d} l$. 
There exist  $\eta\in\Sing_d(1,1;\op)_{d-1}$ and $\nu\in\Sing(2,0;\op)_d$ satisfying
\begin{align}\label{etad}
\partial\eta&=a\circ_1^\op f,\\
\partial\nu&=\eta\circ_1^\op f-(-1)^{d-1} (\eta\circ_1^\op f)\cdot (21)-f\circ_1^\cl l. \label{reld}
\end{align}
 \end{prop}

\begin{proof} The proof is divided into four steps.

\medskip
\noindent{\sl i) Construction of $f$ and $l$}.
The space $\SC_d(1,0;\op)$ is contractible and $f$ denotes either the configuration 
$B((0,\ldots,0,\frac{1}{2});\frac{1}{2})$ in $B^+(0;1)$ or the associated
$0$-singular chain.

The space $\SC_d(2,0;\cl)$ is homotopy equivalent to $S^{d-1}$. Thus, there is a cycle 
$\tilde l\in \Sing_d(2,0;\cl)_{d-1}$ such that $[\tilde l]$
spans $\hsc_d(2,0;\cl)_{d-1}$. The action of the symmetric group corresponds to the antipodal map, so 
$[\tilde l\cdot (21)]=(-1)^d[\tilde l]$. Replacing $\tilde l$ by 
$l=\frac{1}{2} (\tilde l+(-1)^d \tilde l\cdot (21))$ gives a cycle satisfying the required conditions.
Note that $[f\circ_1^\cl l]$ spans $\hsc_d(2,0;\op)_{d-1}$.

\medskip
\noindent{\sl ii) Construction of $a$}.
The space $\SC_d(0,2;\op)$ is  homotopy equivalent to $S^{d-2}$ and a homotopy equivalence is given by the map
$$\begin{array}{cccc}
\pi_\op:& \SC_d(0,2;\op)&\rightarrow& S^{d-2} \\
& \{B^+(x;r),B^+(y;s)\}&\mapsto& \frac{(x-y)}{||x-y||}
\end{array}$$
The map 
$$\begin{array}{cccc}
\varphi_\op:&S^{d-2} &\rightarrow & \SC_d(0,2;\op) \\
&x&\mapsto &  \{B^+(\frac{x}{2};\frac{1}{2}),B^+(-\frac{x}{2};\frac{1}{2})\}
\end{array}$$
is a homeomorphism onto its image.
The composition $\pi_\op\varphi_\op$ is the identity, hence $\varphi_\op$ induces a quasi-isomorphism between 
the singular chain complexes. Let $e_{d-2}$ be a singular chain such that $[e_{d-2}]$ is a generator 
of $H_{d-2}(S^{d-2})$ and let $a=\varphi_\op(e_{d-2})$. The homology class  $[a]$ spans $\hsc_d(0,2;\op)_{d-2}$.
The action of the antipodal map on $S^{d-2}$ corresponds to the action of the symmetric group on $\SC_d(0,2;\op)$
via the map $\varphi_\op$, hence
we can choose $e_{d-2}$ so that $a\cdot (21)=(-1)^{d-1} a$.

\medskip

\noindent{\sl iii) Construction of $\eta$.}
The composition
$$\begin{array}{cccc}
\varphi_\op\circ_1^\op f :&S^{d-2} &\rightarrow & \SC_d(1,1;\op) \\
&x&\mapsto &  \{B(\frac{x}{2}+P_{\frac{1}{4}};\frac{1}{4}),B^+(-\frac{x}{2};\frac{1}{2})\}
\end{array}$$
is a homeomorphism onto its image $Y_\op$. Thus the homology class $[a\circ_1^\op f]$ spans $H_{d-2}(Y_\op)$.
The map $\varphi_\op\circ_1^\op f$ has a lifting
 $$\begin{array}{cccc}
\psi_\op :&D^{d-1} &\rightarrow & \SC_d(1,1;\op) \\
&x=(x_1,\ldots,x_{d-1},x_d\geq 0)&\mapsto &  \{B(\frac{x}{2}+P_{\frac{1}{4}};\frac{1}{4}),
B^+((-\frac{x_1}{2},\ldots,-\frac{x_{d-1}}{2},0);\frac{1}{2})\}
\end{array}$$
and the following diagram is commutative
$$\xymatrix{
H_{d-1}(D^{d-1},S^{d-2})\ar[r]^>>>>>\partial\ar[d]_{H_{d-1}(\psi_\op)}& 
H_{d-2}(S^{d-2})\ar[d]^{H_{d-2}(\varphi_\op\circ_1^\op f)}\\
H_{d-1}(\SC_d(1,1;\op),Y_\op)\ar[r]_<<<<<{\partial}&H_{d-2}(Y_\op)}$$
The horizontal maps are pieces of the homology long exact sequence associated to a pair of spaces and since 
$D^{d-1}$ and $\SC_d(1,1;\op)$ are contractible they are isomorphisms. The map $\varphi_\op\circ_1^\op f$ 
is a homeomorphism from $S^{d-2}$ to $Y_\op$ hence $H_{d-2}(\varphi_\op\circ_1^\op f)$ 
is an isomorphism. As a consequence, the left map in the diagram is an isomorphism.
If $u_{d-1}\in C_{d-1}(D^{d-1})$ is such that $[\partial(u_{d-1})]=[e_{d-2}]$ in $H_{d-2}(S^{d-2})$,
then $[\partial\psi_\op(u_{d-1})]=[a\circ_1^\op f]$. As a conclusion, there exists  
$\eta_{d-1}=\psi_\op(u_{d-1})\in \Sing_{d}(1,1;\op)_{d-1}$ such that 
\begin{equation*}
[\partial \eta_{d-1}]=[a\circ_1^\op f],
\end{equation*}
and there exists $x_{d-1}\in C_{d-1}(Y_\op)$ such that $\partial\eta_{d-1}-\partial x_{d-1}=a\circ_1^\op f$.
Hence the element $\eta=\eta_{d-1}-x_{d-1}$ satisfies Relation (\ref{etad}).

\medskip
\noindent{\sl iv) Proof of Relation (\ref{reld}).} Consider the composition
$$\begin{array}{cccc}
\psi_\op\circ_1^\op f :&D^{d-1} &\rightarrow & \SC_d(2,0;\op) \\
&x&\mapsto &  \{B(\frac{x}{2}+P_{\frac{1}{4}};\frac{1}{4}),B((-\frac{x_1}{2},\ldots,-\frac{x_{d-1}}{2},\frac{1}{4});
\frac{1}{4})\}.
\end{array}$$
Its restriction to $S^{d-2}$ is the map 
$$\begin{array}{cccc}
\varphi_\op(f,f) :&S^{d-2} &\rightarrow & \SC_d(2,0;\op) \\
&x&\mapsto & \{B(\frac{x}{2}+P_{\frac{1}{4}};\frac{1}{4}),B(-\frac{x}{2}+P_{\frac{1}{4}};\frac{1}{4})\}.
\end{array}$$
Recall that $R_{d-1}$ denotes the composite of the $d-1$ 
reflections $r_1,\ldots,r_{d-1}$. Note that
\begin{equation}\label{E:invariance}
\varphi_\op(f,f)\cdot (21)\cdot R_{d-1}=\varphi_\op(f,f).
\end{equation}

%$\varphi_\op(f,f)$ which sends $x\in S^{d-2}$ to the configuration
%$\{B(\frac{x}{2}+P_{\frac{1}{4}};\frac{1}{4}),B(-\frac{x}{2}+P_{\frac{1}{4}};\frac{1}{4})\}$.
The projection
$$\begin{array}{cccc}
\pi_\cl :&\SC_d(2,0;\op)  &\rightarrow & S^{d-1} \\
& \{B(x;r),B(y;s)\}&\mapsto& \frac{(x-y)}{||x-y||}
\end{array}$$
is a homotopy equivalence, and we denote by $J^+:D^{d-1}\rightarrow S^{d-1}$ the composite 
$\pi_\cl\cdot (\psi_\op\circ_1^\op f)$. The embedding  $J^+$ maps $x=(x_1,\ldots,x_{d-1},x_d\geq 0)$
to $\frac{x'}{||x'||}$, where $x'=(x_1,\ldots,x_{d-1},\frac{x_d}{2})$, and
\begin{equation}\label{E:Jplus}
   \pi_\cl\cdot ((\psi_\op\circ_1^\op f)\cdot (21)\cdot R_{d-1}) =J^+\cdot r_d
  \end{equation}

Consider the element $u_{d-1}\in C_{d-1}(D^{d-1})$ defined in step iii).
The homology class
$[J^+(u_{d-1})-(J^+\cdot r_d)(u_{d-1})]$ spans $H_{d-1}(S^{d-1})$. Thus, the homology class
$[\eta_{d-1}\circ_1^\op f-(\eta_{d-1}\circ_1^\op f)\cdot (21)\cdot R_{d-1}]$ spans $\hsc_{d}(2,0;\op)_{d-1}$,
which is also spanned by $[f\circ_1^\cl l]$. Multiplying $l$ by a scalar $\lambda\not=0\in\kfield$, we can assume
\begin{equation*}%\label{R:etafirst}
[\eta_{d-1}\circ_1^\op f-(\eta_{d-1}\circ_1^\op f)\cdot (21)\cdot R_{d-1}]=[f\circ_1^\op l].
\end{equation*}

If  $x\in C_{d-1}(Y_\op)$, then $x\circ_1^\op f-(x\circ_1^\op f)\cdot (21)\cdot R_{d-1}=0$ by Relation (\ref{E:invariance}).
As a consequence, since $\eta=\eta_{d-1}-x_{d-1}$, with $x_{d-1}\in C_{d-1}(Y_\op)$,
\begin{equation}\label{R:etafirst}
[\eta\circ_1^\op f-(\eta\circ_1^\op f)\cdot (21)\cdot R_{d-1}]=[f\circ_1^\op l].
\end{equation}
In addition,  Relation (\ref{etad}) and Equality $a\cdot(21)=(-1)^{d-1}a$ imply
\begin{equation*}%\label{R:etafirst}
\partial(\eta\circ_1^\op f-(-1)^{d-1}(\eta\circ_1^\op f)\cdot (21))=a(f,f)-(-1)^{d-1}a\cdot (21)(f,f)=0.
\end{equation*}

Hence $\eta\circ_1^\op f-(-1)^{d-1}(\eta\circ_1^\op f)\cdot (21)$ is a cycle, and so is
$(\eta\circ_1^\op f)\cdot (21)\cdot R_{d-1}-(-1)^{d-1}(\eta\circ_1^\op f)\cdot (21)$. Let $A$ denote the homology class of the latter cycle.
On the one hand, the space $\SC(2,0;\op)$ is invariant under the action of $R_{d-1}$, 
and the homology of this action is the multiplication by $(-1)^{d-1}$, so that
$A\cdot R_{d-1}=(-1)^{d-1}A$. On the other hand,

$$A\cdot R_{d-1}=[(\eta\circ_1^\op f)\cdot (21)-(-1)^{d-1}(\eta\circ_1^\op f)\cdot (21)\cdot R_{d-1}]=(-1)^d A.$$

As a conclusion $2A=0$, hence $A=0$, because the field is of characteristic different from 2. Together  with Relation (\ref{R:etafirst}) it implies that
\begin{equation*}
[\eta\circ_1^\op f-(-1)^{d-1}(\eta\circ_1^\op f)\cdot (21)]=[f\circ_1^\op l].
\end{equation*}

Hence there exists $\nu$ satisfying 
Relation (\ref{reld}). \end{proof}

The next proposition is analogous to Proposition \ref{P:change} which states that Relations (\ref{etad}) and
(\ref{reld}) do not depend on the generators. 

\begin{prop}\label{P:changed} If $a'\in\Sing(0,2;\op)_{d-2}, f'\in\Sing(1,0;\op)_0$ and 
$l'\in\Sing(2,0;\cl)_{d-1}$ are cycles satisfying 
$[a']=[a], [f']=[f]$ and $[l']=[l]$, then we can assume that $a'\cdot (21)=(-1)^{d-1} a'$ and 
$l'\cdot (21)=(-1)^{d} l'$. Moreover,  there exist $\eta'\in \Sing(1,1;\op)_{d-1}$ 
and $\nu'\in\Sing(2,0;\op)_d$  such that

$$\begin{array}{lcclc}
(\ref{etad}')&\partial\eta'&=&(a'\circ_1^\op f'),  \\
(\ref{reld}')& \partial\nu'&=&\eta'\circ_1^\op f'-(-1)^{d-1} (\eta'\circ_1^\op f')\cdot (21)-f'\circ_1^\cl l'.
\end{array}$$
 \end{prop}
 
 \begin{proof}
  Since the ground field $\kfield$ has characteristic different from 2, replacing $a'$ by 
  $\frac{1}{2}(a'+(-1)^{d-1}a'\cdot (21))$ and $l'$ by 
  $\frac{1}{2}(l'+(-1)^{d}l'\cdot (21))$ gives the first assertion of the proposition. 
  
There exist $a_{d-1},f_1,l_d$ such that
$a'=a+\partial a_{d-1},
f'=f+\partial f_1 
\text{ and } 
l'=l+\partial l_d,$
and using the same trick, we can assume that $a_{d-1}\cdot (21)=(-1)^{d-1}a_{d-1}$ and $l_d\cdot (21)=(-1)^d l_d.$
The elements
\begin{align*}
\eta'=&\eta+(a_{d-1}\circ_1^\op f')+(-1)^{d-2}a\circ_1^\op f_1, \\
\nu'=&\nu-f_1\circ_1^\cl l'-f\circ_1^\cl l_d+(-1)^{d-1}\eta\circ_1^\op f_1-\eta\circ_1^\op f_1\cdot(21)
-a(f_1,f_1),
\end{align*}
satisfy the prescribed relations.
\end{proof}
Mimicking the proof of Theorem \ref{T:main2}, we obtain the following theorem. 

\begin{thm} The $d$-dimensional Swiss-cheese operad is not formal.
\end{thm}

\section{Massey operadic products and non-formality}

In this section we show that the technics developped for proving the non-formality of the Swiss-cheese operad can be generalized to give a criteria
for non-formality of a topological operad. This criteria is similar in spirit to the one applied in \cite[Section 3]{FM08}. For  clarity of exposition
we only treat the uncolored case.

\medskip

For $\C$ a topological operad, we denote by $C(\C)$ its singular chain complex and by $H(\C)$ its homology.
We say that $\C$ is formal if there exists a dg operad $\P$ together with a zig-zag of operad quasi-isomorphisms
$$\xymatrix{ & \P \ar[dl]_{\gamma}\ar[dr]^{\psi} & \\
C(\C) & &H(\C)}$$

\subsection{Massey operadic products}\label{S:massey} Let $\P$ be a dg operad.
 Let $a,b,c$ be cycles in $\P$ having respectively arity $p,q,r$. If there exist $x,y\in \P$ such that 
 $dx=a\circ_i b$ and $dy=b\circ_j c$, then
 $$M=x\circ_{i+j-1} c-(-1)^{|a|} a\circ_i y$$
 satisfies $dM=0$. 
 The class $[M]$ in $H(\P)$ is called a Massey operadic product of type $I$.  
 
 The proof of the following proposition follows the classical proof for triple Massey products.
 
 \begin{prop} With the above notation, the set of Massey operadic products of type $I$  depends only on the classes $[a],[b],[c]$. This set is the coset 
 $[M]+[a]\circ_i H_{|b|+|c|+1}(\P)(q+r-1)+H_{|a|+|b|+1}(\P)(p+q-1)\circ_{i+j-1}[c],$ and is denoted $\langle [a],[b],[c]\rangle_I$.
 \end{prop}

 Similarly, if there exist $u,v\in \P$ such that $du=a\circ_i b$ and $dv=a\circ_j c$ with $i<j$, then
 $$N=u\circ_{j+q-1} c-(-1)^{|b||c|} v\circ_i b$$
 satisfies $dN=0$. The class $[N]$ in $H(\P)$ is called a Massey operadic product of type $II$. 
 
  \begin{prop} With the above notation, the set of Massey operadic product of type $II$ depends only on the classes $[a],[b],[c]$. This set  is the coset 
 $[N]+ H_{|a|+|b|+1}(\P)(p+q-1)\circ_{j+q-1}[c]+ H_{|a|+|c|+1}(\P)(p+r-1)\circ_{i}[b]$ and is denoted $\langle [a],[b],[c]\rangle_{II}$.
 \end{prop}
 
 \begin{defn} In case of existing Massey products $ \langle[a],[b],[c] \rangle$ of type $I$ or $II$, if $0$ does not belong to $ \langle[a],[b],[c] \rangle$ we say that $\P$ admits non-vanishing  Massey operadic products.
\end{defn}

The following proposition is straightforward.

\begin{prop}\label{P:masseymor} Let $\varphi:\P\rightarrow \mathcal{Q}$ be a morphism of dg operads. For $ \langle[a],[b],[c] \rangle$ a Massey operadic product of type $I$ or $II$ in $\P$ one has
$$\varphi_*( \langle[a],[b],[c] \rangle)= \langle\varphi_*([a]),\varphi_*([b]),\varphi_*([c]) \rangle.$$
If $\mathcal{Q}$ admits non-vanishing  Massey operadic product so does $\P$.
\end{prop}
 
 \begin{cor}\label{T:general}
  Let $\C$ be a topological operad. If  $C(\C)$ admits non-vanishing Massey operadic products then $\C$ is not formal.
   \end{cor}

 \begin{proof}
 We assume by contradiction that
 there is a dg operad $\P$ together with a zig-zag of operad quasi-isomorphisms
$$\xymatrix{ & \P \ar[dl]_{\gamma}\ar[dr]^{\psi} & \\
C(\C) & &H(\C)}$$

Let us assume that $C(\C)$ admits a non-vanishing Massey operadic product $ \langle[a],[b],[c] \rangle_I$ of type $I$.
Because $\gamma$ is a quasi-isomorphism, there exists  for every $u\in\{a,b,c\}$ a cycle $u^\P\in \P$ such that
$\gamma_*([u^\P])=[u]$.
Injectivity of $\gamma_*$ implies $[a^\P\circ_i b^\P]=0=[b^\P\circ_j c^\P]$, hence there is a well defined Massey operadic product 
$ \langle[a^\P],[b^\P],[c^\P] \rangle_I$ which is sent to $ \langle[a],[b],[c] \rangle_I$ via $\gamma_*$. By Proposition \ref{P:masseymor}, $0\not\in    \langle[a^\P],[b^\P],[c^\P] \rangle_I$,
hence $\P$ admits non-vanishing Massey operadic products. In $H(\C)$ however all  Massey operadic products vanish and
$$\psi_*( \langle[a^\P],[b^P],[c^P] \rangle_I)= \langle\psi_*[a^\P],\psi_*[b^\P],\psi_*[c^\P] \rangle_I\subset H(\C).$$
The surjectivity of $\psi_*$ implies that $0\in  \langle[a^\P],[b^\P],[c^\P] \rangle_I$, a contradiction. The proof goes the same if we assume that
$C(\C)$ admits a non-vanishing Massey operadic product $ \langle[a],[b],[c] \rangle_{II}$ of type $II$.
\end{proof}

\subsection{Conclusion}

Proposition \ref{P:changed} consists in proving that $0\not\in \langle a;f,f\rangle_{II}$. Indeed, since
$\partial\eta=a\circ_1^\op f$ and $\partial((-1)^{d-1}\eta\cdot(21))=(-1)^{d-1}(a\circ_1^\op f)\cdot(21)=a\circ_2^\op f$,
Propositon \ref{P:changed} shows that the class of 
$$M=\eta\circ_1^\op f-(-1)^{d-1}\eta\cdot(21)\circ_1^\op f=\eta\circ_1^\op f-(-1)^{d-1}\eta\circ_1^\op f\cdot (21)$$
does not vanish because it is equal to $[f\circ_1^\cl l]$. Moreover $\hsc_d(\cl,\op;\op)_{d-1}=0$
implies that $\langle a;f,f\rangle_{II}=[M]$, so that $0\not\in \langle a;f,f\rangle_{II}$.
A colored version of Corollary \ref{T:general} implies that the Swiss-cheese operad is not formal. 
\bigskip

\noindent{\bf Aknowledgement.}
The author would like to thank the organizers of the GDO (Grothendieck-Teichm\"uller Groups, Deformation and Operads) 
program held at the Institute Isaac Newton in 2013 where this result has been obtained and the ANR HOGT.
She would like to thank Damien Calaque, 
Gregory Ginot, Eduardo Hoefel, Victor Turchin and Sasha Voronov for fruitful discussions.

\def\cprime{$'$} \def\cprime{$'$} \def\cprime{$'$} \def\cprime{$'$}
\providecommand{\bysame}{\leavevmode\hbox to3em{\hrulefill}\thinspace}
\providecommand{\MR}{\relax\ifhmode\unskip\space\fi MR }
\providecommand{\MRhref}[2]{%
  \href{http://www.ams.org/mathscinet-getitem?mr=#1}{#2}
}
\providecommand{\href}[2]{#2}


\begin{thebibliography}{10}


\bibitem{BV}
J.M. Boardman and R.M. Vogt, \emph{Homotopy invariant algebraic structures on topological spaces}, 
Lecture Notes in Mathematics, Vol. 347. Springer-Verlag, Berlin-New York, 1973. x+257 pp. 

\bibitem{DGMS}
Pierre Deligne, Phillip Griffiths, John  Morgan  and Dennis Sullivan,
\emph{Real homotopy theory of K\"ahler manifolds}, 
Invent. Math. \textbf{29} (1975), no. 3, 245Ð274. 

\bibitem{Dolgushev11}
Vasily Dolgushev, \emph{Formality theorem for Hochschild cochains via transfer}. Lett. Math. Phys. \textbf{97} (2011), no. 2, 109Ð149.

\bibitem{FM08}
Marisa Fern\' andez and Vincente Mu\~noz, 
\emph{An 8-dimensional non-formal, simply connected, symplectic manifold}. 
Ann. of Math. (2) \textbf{167} (2008), no. 3, 1045Ð1054. 

\bibitem{GS12}
Jeffrey, Giansiracusa and Paolo Salvatore,  \emph{Cyclic operad formality for compactified moduli spaces of genus zero surfaces}. Trans. Amer. Math. Soc. \textbf{364} (2012), no. 11, 5881Ð5911. 

\bibitem{HoeLiv12}
Eduardo Hoefel and Muriel Livernet, \emph{Open-{C}losed homotopy algebras and
  strong homotopy {L}eibniz pairs through {K}oszul operad theory}, Lett. Math.
  Phys. \textbf{101} (2012), no.~2, 195--222.
  
  \bibitem{HoeLiv13}
Eduardo Hoefel and Muriel Livernet, \emph{On the spectral sequence of the Swiss-cheese operad}, Algebraic and Geometric Topology \textbf{13} (2013) 2039-2060.


\bibitem{Kontse99}
Maxim Kontsevich, \emph{Operads and motives in deformation quantization}, Lett.
  Math. Phys. \textbf{48} (1999), no.~1, 35--72, Mosh\'e Flato (1937--1998).
  
\bibitem{Kontse03}
Maxim Kontsevich, \emph{Deformation quantization of Poisson manifolds}, Lett. Math. Phys. \textbf{66} (2003), no. 3, 157-216. 

\bibitem{LV08}
Pascal Lambrechts and Ismar Volic, \emph{Formality of the little $N$-discs operad},  to appear in Memoirs of the AMS, arXiv: 0808.0457.

\bibitem{May}
 J. P. May, \emph{The geometry of iterated loop spaces}, Lectures Notes in Mathematics, Vol. 271. Springer-Verlag, Berlin-New York, 1972. viii+175 pp. 
 
%\bibitem{RT10}
%Yuli Rudyak and Aleksy Tralle, \emph{On {T}hom spaces, {M}assey products, and nonformal symplectic manifolds},
 %Internat. Math. Res. Notices, \textbf{10} (2000), 495-513.
  
  \bibitem{Sullivan77}
Denis Sullivan, \emph{Infinitesimal computations in topology}, Inst. Hautes Etudes Sci. Publ. Math. \textbf{47} (1977), 269-331.

\bibitem{TurWill}
Victor Turchin and  Thomas Willwacher, \emph{Relative (non-)formality of the little cubes operads and the algebraic Schoenflies theorem},
preprint  arXiv:1409.0163.



\bibitem{Voronov99}
Alexander~A. Voronov, \emph{The {S}wiss-cheese operad}, Homotopy invariant
  algebraic structures ({B}altimore, {MD}, 1998), Contemp. Math., vol. 239,
  Amer. Math. Soc., Providence, RI, 1999, pp.~365--373.

\end{thebibliography}
\end{document}